\documentclass[11pt]{amsart}
\usepackage{amsthm}
\usepackage{amssymb}
\usepackage{amsmath}
\input xy
\xyoption{all} 
\usepackage{tikz-cd}
\usepackage{amsthm}
\usepackage{amssymb}
\usepackage{amsmath}
\usepackage[shortlabels]{enumitem}
\usepackage{nicefrac}

\input xy
\xyoption{all} 

\newtheorem{theorem}{Theorem}
\newtheorem*{theorem*}{Theorem}
\newtheorem*{remark*}{Remark}
\newtheorem*{lemma*}{Lemma}

\newtheorem{example}[theorem]{Example}
\newtheorem*{example*}{Example}

\newtheorem*{examples*}{Examples}
\newtheorem*{exercise*}{Exercise}

\newtheorem{definition}[theorem]{Definition}
\newtheorem*{definition*}{Definition}
\newtheorem{lemma}[theorem]{Lemma}

\newtheorem{proposition}[theorem]{Proposition}
\newtheorem{remark}[theorem]{Remark}
\newtheorem{cor}[theorem]{Corollary}
\newtheorem*{cor*}{Corollary}

\newtheorem{question}{Question}

\newtheorem*{conjecture*}{Conjecture}
\numberwithin{theorem}{section}

\renewcommand\iff{%
\ifmmode\text{ if and only if }%
\else if and only if \fi}

\newcommand{\nf}{\nicefrac}
\renewcommand{\and}{\wedge}

\renewcommand{\phi}{\varphi}

\newcommand{\Mod}{\textnormal{Mod}}
\renewcommand{\mod}{\text{mod}}

\newcommand{\Ab}{\text{Ab}}

\newcommand{\rad}{\textnormal{rad}}

\newcommand{\End}{\textnormal{End}}
\newcommand{\Hom}{\textnormal{Hom}}
\newcommand{\Ext}{\textnormal{Ext}}

\newcommand{\Zg}{\textnormal{Zg}}

\newcommand{\pinj}{\textnormal{pinj}}

\newcommand{\im}{\textnormal{im}}

\newcommand{\mcal}[1]{\mathcal{#1}}

\newcommand{\mfrak}[1]{\mathfrak{#1}}
\newcommand{\st}{\ \vert \ }

\newcommand{\pp}{\textnormal{pp}}

\newcommand{\N}{\mathbb{N}}
\newcommand{\Z}{\mathbb{Z}}

\newcommand{\coker}{\text{coker}}

\usepackage{fullpage}

\usepackage{dsfont}
\usepackage{nicefrac}

\newcommand{\Tf}{\text{Tf}}

\newcommand{\Latt}{\text{Latt}}
\newenvironment{pmat}{\left( \begin{smallmatrix}}{\end{smallmatrix} \right)}

\usepackage[pdftex,linktocpage]{hyperref}

\title{Interpretation functors which are full on pure-injective modules with applications to $R$-torsion-free modules over $R$-orders.}
\author{Lorna Gregory}
\address{School of Mathematics, University of East Anglia, Norwich, NR4 7TJ, UK}
\email{Lorna.Gregory@gmail.com}

\thanks{The author was supported by PRIN 2017-Mathematical Logic: models, sets, computability.}



\begin{document}

\begin{abstract}
Let $R,S$ be rings, $\mcal{X}\subseteq \mod$-$R$ a covariantly finite subcategory, $\mcal{C}$ the smallest definable subcategory of $\Mod$-$R$ containing $\mcal{X}$ and $\mcal{D}$ a definable subcategory of $\Mod$-$S$. We show that if $I:\mcal{C}\rightarrow \mcal{D}$ is an interpretation functor such that $I\mcal{X}\subseteq \mod$-$S$ and whose restriction to $\mcal{X}$ is full then $I$ is full on pure-injective modules. We apply this theorem to an extension of a functor introduced by Ringel and Roggenkamp which, in particular, allows us to describe the torsion-free part of the Ziegler spectra of tame B\"ackstr\"om orders. We also introduce the notion of a pseudogeneric module over an order which is intended to play the same role for lattices over orders as generic modules do for finite-dimensional modules over finite-dimensional algebras.
\end{abstract}

\maketitle

\section{Introduction}
Interpretation functors are a uniform additive version of the model theoretic notion of interpretation for (definable subcategories of) module categories. Algebraically, they are
additive functors between definable subcategories of module categories which commute with direct limits and products. They are key tools for studying module categories from the perspective of model theory and for the algebraic study of purity.

When an interpretation functor $I:\mcal{C}\rightarrow \mcal{D}$ is full on pure-injective modules the connection between $\mcal{C}$ and $\mcal{D}$ is much closer than for other interpretation functors (more precisely, the connection between $\mcal{C}$ ``modulo'' $\ker I$ and $I\mcal{C}$). Model theoretically being full on pure-injective modules corresponds, \cite[3.17]{Intmodinmod}, to the interpretation given by $I$ ``preserving all induced structure''. 

The main theorem of this paper (see \ref{fullonfpimpliesfullonpureinjectives}) is the following.  Here, $\mod\text{-}R$ denotes the category of finitely presented (right) $R$-modules.

\begin{theorem*}\label{fullonfpimpliesfullonpureinjectives}
Let $R,S$ be rings, $\mcal{X}\subseteq \mod\text{-}R$ be covariantly finite, $\mcal{C}$ the smallest definable subcategory containing $\mcal{X}$ and $\mcal{D}$ a definable subcategory of $\Mod\text{-}S$. If  $I:\mcal{C}\rightarrow \mcal{D}$ is an interpretation functor and $I$ is full on $\mcal{X}$ and sends modules in $\mod\text{-}R$ to modules in $\mod\text{-}S$ then $I$ is full on pure-injectives.
\end{theorem*}

In general, interpretation functors do not send finitely presented modules to finitely presented modules. However, there are naturally occurring situations when they do. 
For instance, this is the case for $k$-linear interpretation functors between categories of modules over finite-dimensional $k$-algebras (see \ref{noethklin} for a generalisation of this observation). 
Moreover, many functors used in the literature to understand one category of finitely presented modules in terms of another extend to interpretation functors between categories of non-finitely presented modules. Thus, in these situations, our theorem can be applied. 

One such example is Butler's functor, \cite{Butler}, between a large subcategory of lattices over the $\widehat{\Z_2}$-order $\widehat{\Z_2}C_2\times C_2$ to the category of finite-dimensional $\Z/2\Z$-vector spaces with $4$ distinguished subspaces fulfills the hypotheses of our theorem\footnote{One needs to check that the category of ``$b$-reduced'' lattices is a covariantly finite subcategory of $\mod$-$\widehat{\Z_2}C_2\times C_2$}. Using clever ad hoc methods, in \cite{Klein4}, Puninski and Toffalori showed that its extension to an interpretation functor is full on pure-injective $\widehat{\Z_2}$-torsion free $\widehat{\Z_2}C_2\times C_2$-modules. 

The second half of this article, section \ref{app}, will give another such application using a functor defined by Ringel and Roggenkamp in \cite{RR}. Let $R$ be a complete discrete valuation domain with maximal ideal generated by $\pi$ and field of fractions $Q$. Let $\Lambda$ be an $R$-order with $A:=Q\Lambda$ a separable $Q$-algebra. Let $\Gamma$ be a hereditary $R$-order such that $\Lambda\subseteq \Gamma\subseteq A$ and let $I\subseteq \rad(\Lambda)$ be an ideal of $\Gamma$ such that $\pi^n\in I$ for some $n\in\N$. Let \[D:=
                                                          \begin{pmat}
                                                            \Lambda/I & \Gamma/I \\
                                                            0 & \Gamma/I \\
                                                          \end{pmat}.\]
Ringel and Roggenkamp defined a full (but not faithful) additive functor from the category of $\Lambda$-lattices to the category of finitely generated $D$-modules. This functor extends to an interpretation functor $\mathbb{F}:\Tf_\Lambda\rightarrow \Mod\text{-}D$ where $\Tf_\Lambda$ denotes the category of $R$-torsion-free $\Lambda$-modules. Since categories of $\Lambda$-lattices are covariantly finite and the smallest definable subcategory containing the $\Lambda$-lattices is $\Tf_\Lambda$, Theorem \ref{fullonfpimpliesfullonpureinjectives} implies that $\mathbb{F}$ is full on pure-injectives (see \ref{RRpureinj}). The kernel of $\mathbb{F}$ consists exactly of the $R$-divisible $R$-torsion-free $\Lambda$-module i.e. the restrictions of  $Q\Lambda$-modules.

A useful property of interpretation functors which are full on pure-injective modules is that they induce homeomorphisms between parts of Ziegler spectra. The Ziegler spectrum of a ring $S$, denoted $\Zg_S$, is the set of isomorphism types of indecomposable pure-injective $S$-modules with closed sets those of the form $\mcal{D}\cap\Zg_S$ where $\mcal{D}$ is a definable subcategory. For $\mcal{C}$ a definable subcategory of $\Mod$-$S$, we write $\Zg(\mcal{C})$ for the subspace $\Zg_S\cap \mcal{C}$.
Prest showed in \cite[3.19]{Intmodinmod} that if $I:\mcal{C}\rightarrow \mcal{D}$ is an interpretation functor which is full on pure-injective modules then the map $N\mapsto IN$ induces a homeomorphism from the open subset $\Zg(\mcal{C})\backslash \ker I$ and its image in $\Zg(\mcal{D})$, which is closed. The ($R$-)torsion-free part $\Zg_\Lambda^{tf}:=\Zg(\Tf_\Lambda)$ of the Ziegler spectrum of $\Lambda$ was studied in \cite{tfpartRG}, \cite{Klein4}, \cite{Tfpart} and \cite{Maranda}. Under our hypotheses on $\Lambda$, every $R$-torsion-free $\Lambda$-module decomposes as $D\oplus N$ where $D$ is $R$-divisible and $N$ is $R$-reduced i.e. $\cap_{i\in\N}N\pi^i=\{0\}$. Thus, \ref{RRhomeomorphism}, the functor of Ringel and Roggenkamp induces a homeomorphism from the open subset of $R$-reduced modules in $\Zg_\Lambda^{tf}$ to a closed subset of $\Zg_D$. Note that since $Q\Lambda$ is semi-simple there are only finitely many points in $\Zg_\Lambda^{tf}$ which are not $R$-reduced.

In subsection \ref{ZgFgeneral} we review the general features of $\Zg_\Lambda^{tf}$ and explain how to use $\mathbb{F}$ to completely understand the topology on $\Zg_\Lambda^{tf}$ in terms of $\Zg_D$.

Generic modules, i.e. indecomposable non-finitely presented modules which are finite length as modules over their endomorphism rings, play an important role in the representation theory of Artin algebras. However, if an $R$-torsion-free $\Lambda$-module is finite length over its endomorphism ring then it is $R$-divisible. So, although the indecomposable divisible modules are important they do not play the same role for lattices over orders as generic modules play for finitely presented modules over Artin algebras.
 We propose the notion of pseudogeneric $R$-torsion-free $\Lambda$-modules as a replacement for generic modules over Artin algebras. In subsection \ref{SecPsG}, after proving some routine results about pseudogenerics in analogy with generic modules, we show that Ringel and Roggenkamp's functor $\mathbb{F}$ sends pseudogeneric modules to generic modules. Moreover, we show that, under a variety of conditions, $\mathbb{F}$ gives a bijection between the generic modules in $\mcal{D}$ and the pseudogeneric $\Lambda$-modules. As a consequence, we show that tame B\"ackstr\"om orders of infinite representation type have finitely many pseudogeneric $R$-torsion-free $\Lambda$-modules.

When $\Lambda$ is a tame B\"ackstr\"om order of infinite lattice type, $D$ is a tame hereditary algebra. In subsection \ref{tamebkzg}, we use the description of the Ziegler spectrum of tame hereditary algebras and the results in subsection \ref{ZgFgeneral}, to give a description of $\Zg_\Lambda^{tf}$ when $\Lambda$ is a tame B\"ackstr\"om order of infinite lattice type.

Since the paper is relatively short, we present necessary background material when it is needed. For a general introduction to the algebraic study of Model Theory of Modules see \cite{PrestBluebook} or \cite{PSL}. 

\section{Interpretation functors which are full on pure-injectives}\label{Fullonpi}
\noindent
A (right) \textbf{pp-$n$-formula} is a formula in the language of (right) $R$-modules of the form \[\exists \mathbf{y}\ (\mathbf{x}\mathbf{y})A=0\] where $\mathbf{x}$ is an $n$-tuple of variables, $\mathbf{y}$ is a finite tuple of variables and $A$ is an appropriately sized matrix with entries from $R$. For $M\in\Mod\text{-}R$ and $\phi$ a pp-$n$-formula we write $\phi(M)$ for the solution set of $\phi$ in $M$. Note that this is a subgroup of $M^n$. After identifying pp-$n$-formula whose solution sets are equal in all $R$-modules, the set of pp-$n$-formulae becomes a lattice under inclusion of solution sets, i.e. $\psi\leq \phi$ if $\psi(M)\subseteq \phi(M)$ for all $M\in\Mod\text{-}R$. We denote this lattice $\pp_R^n$. For $X$ a class of $R$-modules, we write $\pp_R^n(X)$ for the quotient of $\pp_R^n$ under the equivalence relation $\phi\sim_X\psi$ if $\phi(M)=\psi(M)$ for all $M\in X$. 
A \textbf{pp-pair}, written $\nf{\phi}{\psi}$, is a pair of pp formulae $\phi,\psi$ such that $\phi(M)\supseteq \psi(M)$ for all $M\in \Mod\text{-}R$.

An embedding $i:A\rightarrow B$ of $R$-modules is \textbf{pure} if for all $\phi\in\pp_R^1$, $\phi(B)\cap iA=i\phi(A)$. An $R$-module $M$ is \textbf{pure-injective} if every pure-embedding $i:M\rightarrow B$ splits. Equivalently, \cite[4.3.11]{PSL}, $N$ is pure-injective if and only if it is \textbf{algebraically compact}. That is, for all $n\in\N$, if for each $i\in \mcal{I}$, $\mathbf{a_i}\in N$ is an $n$-tuple and $\phi_i$ is a pp-$n$-formula then $\bigcap_{i\in \mcal{I}}\mathbf{a_i}+\phi_i(N)=\emptyset$ implies there is some finite subset $\mcal{I}'$ of $\mcal{I}$ with $\bigcap_{i\in \mcal{I}'}\mathbf{a_i}+\phi_i(N)=\emptyset$. We will use the formally stronger equivalent condition (see \cite[4.2.3]{PSL}) given in the following proposition.

\begin{proposition} Let $M$ be a pure-injective $R$-module. Let $(x_j)_{j\in J}$ be a sequence of variables and for each finite subset $E\subseteq J$, write $\mathbf{x}_E$ for the tuple of variables $x_j$ with $j\in E$. For $i\in I$, let $\phi_i$ be a pp formula in variables $\mathbf{x}_{E_i}$ and $\mathbf{a}_i$ a tuple of elements of $M$ indexed by $E_i$.  If for each finite subset $F\subseteq I$ there exists a sequence of elements $(m_j)_{j\in J}$ in $M$ such that $m_{E_i}\in \mathbf{a}_i+\phi_i(M)$ for all $i\in F$ then there exists a sequence of elements $(m_j)_{j\in J}$ in $M$ such that $m_{E_i}\in \mathbf{a}_i+\phi_i(M)$ for all $i\in I$.
\end{proposition}

A full subcategory $\mcal{D}\subseteq \Mod\text{-}R$ is \textbf{definable} if it is closed under products, direct limits and pure-submodules. Equivalently, \cite[3.4.7]{PSL}, a full subcategory $\mcal{D}\subseteq \Mod\text{-}R$ is definable if there exist pp-pairs $\nf{\phi_i}{\psi_i}$ for $i\in I$ such that the modules $M\in \mcal{D}$ are exactly those with $\phi_i(M)=\psi_i(M)$ for all $i\in I$. For any class of $R$-modules $X$, we write $\langle X\rangle$ for the smallest definable subcategory containing $X$.

A full subcategory $\mcal{X} \subseteq \mod\text{-}R$ which is closed under isomorphism, finite direct sums and taking direct summands is \textbf{covariantly finite} in $\mod\text{-}R$ if every $M\in \mod\text{-}R$ has a \textbf{left $\mcal{X}$-approximation} i.e. a homomorphism $f_M:M\rightarrow M_\mcal{X}$ with $M_\mcal{X}\in \mcal{X}$ such that all homomorphisms $g:M\rightarrow L$ with $L\in\mcal{X}$, factor through $f_M$.

For any $\mcal{X}\subseteq \Mod\text{-}R$, let $\overrightarrow{\mcal{X}}$ denote the full subcategory of modules which are direct limits of modules in $\mcal{X}$.

\begin{theorem}\cite[3.4.37]{PSL}
An additive (full) subcategory $\mcal{X}\subseteq \mod\text{-}R$ is covariantly finite in $\mod\text{-}R$ if and only if $\overrightarrow{\mcal{X}}$ is a definable subcategory.
\end{theorem}


Let $\mcal{C}\subseteq \Mod\text{-}R$ and $\mcal{D}\subseteq \Mod\text{-}S$ be definable subcategories.  
Let $\phi(\mathbf{x})/\psi(\mathbf{x})$ be a pp-$m$-pair over $R$ and for each $s\in S$, let $\rho_s(\overline{x},\overline{y})$ be a pp-$2m$-formula such that for each $M\in\mcal{C}$, the solution set $\rho_s(M,M)\subseteq M^m\times M^m$ defines an endomorphism $\rho_s^M$ of the abelian group $\phi(M)/\psi(M)$ and such that $\phi(M)/\psi(M)$ is a $S$-module in $\mcal{D}$ when for all $s\in S$, the action of $s$ on $\phi(M)/\psi(M)$ is given by $\rho^M_s$. 
In this situation, $(\nf{\phi}{\psi};(\rho_s)_{s\in S})$ defines an additive functor $I:\mcal{C}\rightarrow \mcal{D}$. Following \cite{Intmodinmod}, we call any functor equivalent to one defined in this way an \textbf{interpretation functor}. 
We will refer to $\nf{\phi}{\psi}$ as the underlying pp-pair of $I$. Of course this is not a well-defined notion as the same interpretation functor will have many different underlying pp-pairs.

The following theorem, due to Prest in full generality, and Krause in a special case which includes the cases in which we are interested in this paper, gives a completely algebraic characterisation of interpretation functors.

\begin{theorem}\cite[25.3]{DefAddCats}\cite[7.2]{ExDefCats}
An additive functor $I:\mcal{C}\rightarrow \mcal{D}$ is an interpretation functor if and only if $I$ commutes with direct limits and product.
\end{theorem}

Note that we will technically only use the forward direction of the previous theorem which follows easily from the definition of an interpretation functor.

Interpretation functors preserve pure-embeddings and pure-injectivity. Using the definition of an interpretation functor, solution sets of  pp formulae in $IN$ for $N\in\mcal{C}$ can be translated into solution sets of pp formulae in $N$. From this it will follow that pure-embeddings are preserved. It can be seen that interpretation functors preserve algebraic compactness, and hence pure-injectivity, by translating systems of cosets of solution sets of pp formulae for $IN$  into systems of cosets of solution sets of pp formulae for $N$ via $I$.

\begin{theorem}\label{fullonfpimpliesfullonpureinjectives}
Let $R,S$ be rings, $\mcal{X}\subseteq \mod\text{-}R$ be covariantly finite, $\mcal{C}$ the smallest definable subcategory containing $\mcal{X}$ and $\mcal{D}$ a definable subcategory of $\Mod\text{-}S$. If  $I:\mcal{C}\rightarrow \mcal{D}$ is an interpretation functor and $I$ is full on $\mcal{X}$ and sends modules in $\mod\text{-}R$ to modules in $\mod\text{-}S$ then $I$ is full on pure-injectives.
\end{theorem}

Before starting the proof of the theorem, we briefly discuss its hypotheses. If $\mcal{X}\subseteq \mod\text{-}R$ is covariantly finite and $\mcal{C}$ is the smallest definable subcategory containing it then $\mcal{C}$ is locally finitely presented and the finitely presented objects of $\mcal{C}$ are exactly those in $\mcal{X}$ (see \cite[4.1]{locfpaddcat}). So, we are in the situation of the more general theorem \cite[3.12]{Sam}. Hence we have the following.

\begin{remark}
If $I$ in \ref{fullonfpimpliesfullonpureinjectives} is faithful when restricted to $\mcal{X}$, as well as full, then $I$ is full (and faithful) on $\mcal{C}$.
\end{remark}

As remarked in the introduction, the above statement does not hold when $I$ restricted to $\mod\text{-}R$ is not faithful as the next example shows.

\begin{example}
Let $R$ be a discrete valuation domain with maximal ideal generated by $\pi$ which is not complete and let $\widehat{R}$ be its completion. Let $\mcal{X}$ be the category of torsion-free finitely presented $R$-modules. Then $\mcal{X}$ is covariantly finite in $\mod\text{-}R$ and the category $\Tf_R$ of torsion-free $R$-modules is the smallest definable subcategory containing $\mcal{X}$.  Let $I:\Tf_R\rightarrow \Mod\text{-}R/\pi R$ be defined by sending $M\in\Tf_R$ to $M/M\pi$ i.e. $I$ is the interpretation functor with underlying pp-pair $\nf{x=x}{\exists y x=y\pi}$. All non-zero modules in $\mcal{X}$ are isomorphic to $R^n$ for some $n\in\N$, in particular they are all projective $R$-modules. From this it follows that $I$ is full on $\mcal{X}$.

Since $\widehat{R}$ is indecomposable as an $R$-module, $R$ is not a direct summand of $\widehat{R}$. All submodules of $R$ as a module over itself are isomorphic to $R$. Thus, every non-zero homomorphism $f:\widehat{R}\rightarrow R$ factors through a surjective homomorphism $g:\widehat{R}\rightarrow R$. Since $R$ is projective, $g$ is split. This contradicts the fact that $\widehat{R}$ is indecomposable. Therefore $\Hom_R(\widehat{R},R)=0$. On the other hand, $I\widehat{R}\cong IR$, so $\Hom_{R/R\pi}(I\widehat{R},IR)\neq 0$. Therefore $I$ is not full.
\end{example}

In general interpretation functors do not send finitely presented modules to finitely presented modules. There are however a few useful situations when they do.
\begin{remark}\label{noethklin}
Let $k$ be a commutative noetherian ring and let $R,S$ be $k$-algebras which are finitely generated as $k$-modules. If an interpretation functor $I:\mcal{C}\rightarrow\mcal{D}$ is $k$-linear then it sends finitely presented $R$-modules to finitely presented $S$-modules.
\end{remark}
\begin{proof}
Our hypothesis on $R$ and $S$ means that an $R$-module (respectively $S$-module) is finitely presented as an $R$-module (respectively $S$-module) if and only if it is finitely presented as a $k$-module. Let $\phi/\psi$ be the underlying pp-pair of $I$. Suppose $M$ is a finitely presented $R$-module. Then $\phi(M)$ and $\psi(M)$ are $k$-submodules of $M^n$ with $k$-module structure inherited from $M$. Thus, since $k$ is noetherian, they are finitely presented $k$-submodules of $M^n$. Hence $\phi(M)/\psi(M)$ is a finitely presented $k$-module with structure of a $k$-module inherited from $M$. Therefore, it is enough to note that since $I$ is $k$-linear, the $k$-module structure on $IM$ is the same as that on $\phi(M)/\psi(M)$ inherited from $M$.
\end{proof}

\begin{lemma}\label{rightsidefull}
Let $I$, $\mcal{X}$ and $\mcal{C}$ be as in \ref{fullonfpimpliesfullonpureinjectives}. For all $L\in \mcal{X}$ and $N\in \mcal{C}$, the map \[\Hom_R(L,N)\rightarrow \Hom_S(IL,IN) \text{\ \ defined by \ \ } f\mapsto If\] is surjective.
\end{lemma}
\begin{proof}
Let $\eta:\Hom_R(L,-)\rightarrow \Hom_S(IL,I(-))$ be the natural transformation in $(\mcal{C},\Ab)$ with component maps $\eta_M(f)=If$ for $f\in \Hom_R(L,M)$. Since $IL$ is finitely presented and $I$ commutes with direct limits, $\Hom_R(IL,I(-))$ commutes with direct limits. Since $I$ is full on $\mcal{X}$, $\eta_M$ is an epimorphism for all $M\in \mcal{X}$. Direct limits are exact (i.e. a direct limit of exact sequences is an exact sequence) and all modules in $\mcal{C}$ are a direct limit of modules in $\mcal{X}$. Therefore $\eta_N$ is surjective for all $N\in\mcal{C}$ as required.
\end{proof}

If $\mathbf{m}$ is an $n$-tuple of elements from a module $M$ then the \textbf{pp-type} of $\mathbf{m}$ is the set of pp-$n$-formulas $\phi$ such that $\mathbf{m}\in\phi(M)$. If $M\in \mod\text{-}R$ and $\mathbf{m}$ is an $n$-tuple of element from $M$ then, \cite[1.2.6]{PSL}, there exists $\phi\in\pp_R^n$ such that which generates the pp-type of $\mathbf{m}$ as a filter in $\pp_R^n$, i.e.  $\psi$ is in the pp-type of $\mathbf{m}$ if and only if $\psi\geq \phi$. In this situation, for any $L\in \Mod$-$R$, $\mathbf{l}\in\phi(L)$ if and only if there exists a homomorphism $f:M\rightarrow L$ with $f(\mathbf{m})=\mathbf{l}$.

\begin{lemma}\label{morphkerint}
Let $A,B\in\mod\text{-}R$ and $\delta:A\rightarrow B$. Let $L\in\mod\text{-}R$, let $\mathbf{c}\in L$ be an $n$-tuple which generates $L$ and let $N\in\Mod\text{-}R$. For each $\alpha\in\Hom(A,L)$, there exists $\chi^\alpha(\mathbf{x})\in\pp_R^n$ such that
\[\chi^\alpha(N)=\{\epsilon(\mathbf{c})\in N \st  \epsilon\in \Hom(L,N) \text{ and there exists }\beta\in \Hom(B,N) \text{ such that }\epsilon\circ \alpha=\beta\circ \delta  \}\]
\end{lemma}
\begin{proof}
Let $\mathbf{a}:=(a_1,\ldots,a_m)$ generate $A$ and let $\mathbf{b}$ generate $B$. Let $\sigma$ generate the pp-type of $\mathbf{c}$ and $\tau$ generate the pp-type of $\mathbf{b}$. For $1\leq j\leq m$, let $\mathbf{r}_j$ and $\mathbf{s}_j$ be tuples of elements from $R$, of appropriate lengths, such that $\delta(a_j)=\mathbf{b}\cdot \mathbf{r}_j$ and $\alpha(a_j)=\mathbf{c}\cdot \mathbf{s}_j$. Let $\chi^{\alpha}(\mathbf{x})$ be the pp formula
\[\exists y \ \tau(\mathbf{y})\wedge\sigma(\mathbf{x})\wedge \bigwedge_{j=1}^m \mathbf{x}\cdot \mathbf{s}_j=\mathbf{y}\cdot \mathbf{r}_j.\]

Suppose $\mathbf{m}\in\chi^\alpha(N)$. Since $\mathbf{m}\in \sigma(N)$, there exists $\epsilon: L\rightarrow N$ such that $\mathbf{m}=\epsilon(\mathbf{c})$. Let $\mathbf{m}'\in N$ be such that $\mathbf{m}'\in\tau(N)$ and $\mathbf{m}\cdot \mathbf{s}_j=\mathbf{m}'\cdot\mathbf{r}_j$. Since $\mathbf{m}'\in \tau(N)$, there exists $\beta:B\rightarrow N$ such that $\beta(\mathbf{b})=\mathbf{m}'$. Then, for each $1\leq j\leq m$,
\[\epsilon(\alpha(a_j))=\epsilon(\mathbf{c}\cdot\mathbf{s}_j)=\epsilon(\mathbf{c})\cdot\mathbf{s}_j=\mathbf{m}\cdot \mathbf{s}_j=\mathbf{m}'\cdot\mathbf{r}_j=\beta(\mathbf{b})\cdot\mathbf{r}_j=\beta(\mathbf{b}\cdot\mathbf{r}_j)=\beta(\delta(a_j)).\] So, since $\mathbf{a}$ generates $A$, $\epsilon\circ \alpha=\beta\circ \delta$.

Now suppose that $\epsilon\in \Hom(L,N)$ and $\beta\in \Hom(B,N)$ are such that $\epsilon\circ \alpha=\beta\circ \delta$. Then $\epsilon (\mathbf{c})\in\sigma(N)$ and $\beta(\mathbf{b})\in\tau(N)$. Then
\[\epsilon(\mathbf{c})\cdot\mathbf{s}_j=\epsilon(\mathbf{c}\cdot\mathbf{s}_j)=\epsilon(\alpha(a_j))=\beta(\delta(a_j))=\beta(\mathbf{b}\cdot\mathbf{r}_j)=\beta(\mathbf{b})\cdot \mathbf{r}_j.\qedhere\]
\end{proof}

\begin{lemma}
Let $I$, $\mcal{X}$ and $\mcal{C}$ be as in \ref{fullonfpimpliesfullonpureinjectives}. There exist $A,B\in\mcal{X}$ and $\delta:A\rightarrow B$ such that $I$ composed with the forgetful functor to $\Ab$ is the cokernel of the natural transformation $\Hom_R(\delta,-)$ in $(\mcal{C},\Ab)$.
\end{lemma}
\begin{proof}
Abusing notation, we will write $I$ for $I$ composed with the forgetful functor from $\mcal{D}$ to $\Ab$. The functor $I:\mcal{D}\rightarrow \Ab$ is defined by a pp-pair, the same pp-pair defines a functor $\Mod\text{-}R\rightarrow \Ab$. Thus there exists $C,D\in \mod\text{-}R$ and $\epsilon:C\rightarrow D$ such that $I$ is the cokernel of $\Hom_R(\epsilon,-)$ (see, for instance, \cite[\S 2.1 Lemma 1]{InfdimCB}). Let $\pi_C:C\rightarrow A$ be the left $\mcal{X}$-approximation of $C$. Let
\[\xymatrix@=20pt{
  D  \ar[r]^{\lambda} & P  \\
  C \ar[r]^{\pi_C}\ar[u]^{\epsilon} & A\ar[u]_{\lambda'}   }\]
be the pushout of $\pi_C:C\rightarrow A$ and $\epsilon: C\rightarrow D$. Let $\pi_P:B\rightarrow P$ be a left $\mcal{X}$-approximation of $P$ and let $\delta=\pi_P\circ \lambda'$.

We will now show that the natural transformation $\eta$ induced by the following diagram is an isomorphism.
\[\xymatrix@R=30pt@C=60pt{
  \Hom(D,-)  \ar[r]^{\Hom(\epsilon,-)} & \Hom(C,-)  \ar[r] & I  \ar[r] & 0  \\
  \Hom(B,-) \ar[r]^{\Hom(\delta,-)}\ar[u]_{\Hom(\lambda\circ\pi_P,-)} & \Hom(A,-)\ar[r]\ar[u]_{\Hom(\pi_C,-)} &  \coker\Hom(\delta,-)\ar[r]\ar@{-->}[u]_{\eta} & 0  }\]

By definition of left approximation, for all $M\in\mcal{C}$, $\Hom(\pi_C,M)$ is surjective. Therefore $\eta_M$ is surjective for all $M\in\mcal{C}$.
We now show that $\eta_M$ is injective for all $M\in\mcal{C}$. Take $M\in\mcal{C}$ and $g:A\rightarrow M$. Suppose that $g\circ \pi_C\in \im\Hom(\epsilon, M)$. By the defining property of pushouts, there exists $h:P\rightarrow M$ such that $h\circ\lambda'=g$. Since $\pi_P$ is a left $\mcal{X}$-approximation of $P$, there exists $h':B\rightarrow M$ such that $h=h'\circ \pi_B$. Therefore $h'\circ \delta=h'\circ \pi_B\circ\lambda'=g$ as required.
%
%
%
%
%
%
%
%
%
%
%
%
%
%
%
%
%
%
%
%
%
%
\end{proof}

\begin{proof}[Proof of theorem \ref{fullonfpimpliesfullonpureinjectives}]
Suppose that $N\in\mcal{C}$ is pure-injective, $N'\in\mcal{C}$ and $g\in \Hom(IN',IN)$. We aim to show that there exists $h\in \Hom_R(N',N)$ such that $Ih=g$.

Since $\mcal{X}$ is covariantly finite and $\mcal{C}$ is the smallest definable subcategory containing $\mcal{X}$, all modules in $\mcal{C}$ are direct limits of modules in $\mcal{X}$. Fix $J$ a directed partially ordered set, $L_i\in \mcal{X}$ and for each $i\leq j\in J$, $\sigma_{ij}:L_i\rightarrow L_j$ such that $N'$ together with the canonical maps $f_i:L_i\rightarrow N'$ is the direct limit of this directed system.

We will show below that there exists a set of homomorphisms $h_i:L_i\rightarrow N$ indexed by $J$ such that for all $i\leq j\in J$, $h_j\sigma_{ij}=h_i$ and $Ih_i=g\circ If_i$.  Since $h_j\sigma_{ij}=h_i$ for all $i\leq j$, there exists a homomorphism $h:N'\rightarrow N$ such that $hf_i=h_i$ for all $i\in I$. Thus $IhIf_i=Ihf_i=g\circ If_i$ for all $i\in J$. Since $I$ commutes with direct limits, this implies that $Ih=g$. So $h$ is the required homomorphism.

For each $i\in J$, let $\mathbf{c}_i$ generate $L_i$ and let
\[H_i:=\{\epsilon:L_i\rightarrow N \st g\circ If_i=I\epsilon\}\subseteq \Hom(L_i,N).\]

By \ref{rightsidefull}, each $H_i$ is non-empty. For each $i\in J$, pick $\epsilon_i\in H_i$. Note that $H_i:= \epsilon_i+\{\epsilon\in \Hom_R(L_i,N) \st I\epsilon=0\}$.

Since $I$ is an interpretation functor, there exists $A,B\in \mod\text{-}R$ and $\delta:A\rightarrow B$ such that, as a functor from $\mcal{C}$ to $\Ab$ (i.e. after composing with the forgetful functor), $I$ has a presentation
\[\xymatrix@C=0.8cm{
   \Hom_R(B,-) \ar[rr]^{\Hom_R(\delta,-)} && \Hom_R(A,-) \ar[rr] && I \ar[r] & 0 }.\]

Since $\mcal{X}\subseteq\mod\text{-}R$ is covariantly finite, we may assume that $A,B\in \mcal{X}$. Now $I\epsilon=0$ if and only if for all $\alpha\in \Hom(A,L)$ there exists $\beta\in \Hom(B,N)$ such that $\epsilon\circ \alpha=\beta\circ\delta$. For each $i\in J$, let $\chi_i^\alpha$ be $\chi^\alpha$ from \ref{morphkerint}, where $L:=L_i$ and $\mathbf{c}:=\mathbf{c}_i$. Thus, $I\epsilon=0$ if and only if $\epsilon(\mathbf{c}_i)\in\chi^\alpha(N)$ for all $\alpha\in \Hom(A,L_i)$.

Now for all $i\leq j\in J$, let $\mathbf{t}_{ij}\in R^{l_j}$ be such that $\sigma_{ij}(\mathbf{c}_{i})=\mathbf{c}_{j}\cdot \mathbf{t}_{ij}$.

Consider the system of linear equations and cosets of pp-definable subsets
\[\mathbf{x}_i\in \epsilon_i(\mathbf{c}_i)+\chi^{\alpha}_{i}(N)\tag*{$(1)_i^\alpha$}\]
for $\alpha\in \Hom_R(M,L_i)$ and
\[\mathbf{x}_i=\mathbf{x}_j\mathbf{t}_{ij}\tag*{$(2)_{ij}$}.\]

Suppose $(\mathbf{m}_i)_{i\in J}$ is a solution to this system in $N$. Since $\mathbf{m}_i-\epsilon_i(\mathbf{c}_i)\in\chi^\alpha_i(N)$, there is a (unique) homomorphism $h_i':L_i\rightarrow N$ defined by $h_i'(\mathbf{c}_{i})=\mathbf{m}_i-\epsilon_i(\mathbf{c}_i)$. By the argument above $Ih_i'=0$. Therefore $h_i:=h_i'+\epsilon_i\in H_i$. The second condition $(2)_{ij}$ implies that for all $i\leq j$, $h_j\circ\sigma_{ij}=h_i$ since
\[h_i(\mathbf{c}_i)=\mathbf{m}_i=\mathbf{m}_j\mathbf{t}_{ij}=h_j(\mathbf{c}_j)\cdot \mathbf{t}_{ij}=h_j(\mathbf{c}_j\cdot\mathbf{t}_{ij})=h_j\circ\sigma_{ij}(\mathbf{c}_i).\]

Since $N$ is pure-injective, in order to show that the system of equations has a solution, we just need to show that it is finitely solvable. Let $J_0\subseteq J$ be a finite subset of $J$. Since $J$ is directed, by adding an element to $J_0$, we may assume that there is a $k\in J_0$ such that $i\leq k$ for all $i\in J_0$.

Let $\mathbf{m}_k:=\epsilon_k(\mathbf{c}_k)$ and for each $i\in J_0$, let $\mathbf{m}_i=\mathbf{m}_k\cdot\mathbf{t}_{ik}$. By definition $\mathbf{m}_k$ is in  $\epsilon_k(\mathbf{c}_k)+\chi^{\alpha}_{k}(N)$ for all $\alpha\in \Hom_R(M,L_k)$. Now $I(\epsilon_k\circ \sigma_{ik})=I(\epsilon_k)\circ I(\sigma_{ik})=g\circ I(f_k)\circ I(\sigma_{ik})=g\circ I(f_k\sigma_{ik})=g\circ I(f_i).$ So $\epsilon_k\circ \sigma_{ik}\in H_i$. Thus $\epsilon_k\circ\sigma_{ik}(\mathbf{c}_i)=\mathbf{m}_i$ satisfies $\epsilon_i(\mathbf{c}_i)+ \chi^{\alpha}_{k}(N)$ for all $\alpha\in \Hom_R(M,L_k)$. Thus we have shown that the system is finitely satisfiable.
\end{proof}

Now that we have proved the main theorem of this section we recall some consequences of an interpretation functor being full on pure-injectives which we will use in the next section to compute dimensions on pp-lattices of B\"ackstr\"om orders and to describe the Ziegler spectra of tame B\"ackstr\"om orders.

The (right) Ziegler spectrum, $\Zg_R$, of $R$ is a topological space whose set of points is the set $\pinj_R$ of isomorphism classes of indecomposable pure-injective modules and which has basis of open sets given by
\[\left(\nf{\phi}{\psi}\right):=\{M\in\pinj_R \st \phi(M)\supsetneq\psi(M)\and\phi(M)\}\]
where $\phi,\psi\in\pp_R^1$. The open sets $\left(\nf{\phi}{\psi}\right)$, and $\Zg_R$ itself, are compact.

 There is a bijective correspondence between the definable subcategories of  $\Mod\text{-}R$ and the closed subsets of closed subsets of $\Zg_R$ given by sending a definable subcategory $\mcal{D}\subseteq \Mod\text{-}R$ to $\pinj_R\cap \mcal{D}$ and sending a closed subset $C\subseteq \Zg_R$ to $\langle C\rangle\subseteq \Mod\text{-}R$. For $\mcal{C}$ a definable subcategory of $\Mod\text{-}R$, we write $\Zg(\mcal{C})$ for the $\Zg_R\cap \mcal{C}$ equipped with the subspace topology.

Note that if $I:\mcal{C}\rightarrow \mcal{D}$ is an interpretation functor then $\ker I\cap \Zg(\mcal{C})$ is the closed subset $\Zg(\mcal{C})\backslash\left(\nf{\phi}{\psi}\right)$ where $\nf{\phi}{\psi}$ is the underlying pp-pair of $I$. The image of a definable subcategory $\mcal{C}'\subseteq \mcal{C}$ under an interpretation functor $I$ is not in general a definable subcategory. However, it follows from \cite[3.4.7 (iii)]{PSL} that the closure of $I\mcal{C}'$ under pure-submodules is a definable subcategory of $\mcal{D}$.\looseness=-1

\begin{theorem}\cite[3.19]{Intmodinmod}\cite[18.2.27]{PSL}\label{fpiZghomeo}
Let $\mcal{C},\mcal{D}$ be definable subcategories and let $I:\mcal{C}\rightarrow \mcal{D}$ be an interpretation functor. If $I$ is full on pure-injectives then the assignment $N\mapsto IN$ induces a homeomorphism from the open subset $\Zg(\mcal{C})\backslash\ker I$ of $\Zg(\mcal{C})$ and the closed subset $I\mcal{C}\cap \Zg(\mcal{D})$ of $\Zg(\mcal{D})$.
\end{theorem}

Let $\mcal{L}$ be a non-empty class of lattices closed under sublattices and quotient lattices. The $\mcal{L}$-dimension, denoted $\mcal{L}\text{-}\dim L$, of a bounded modular lattice $L$ measures the ordinal number of times it takes to reach the one element lattice by iteratively collapsing intervals in $\mcal{L}$. The one element lattice is formally given $\mcal{L}$-dimension $-1$ and a lattice $L$ has $\mcal{L}$-dimension $0$ if and only if there exist $\top=a_{n+1}>a_{n}>\ldots>a_0=\bot$ such that $[a_i,a_{i+1}]\in\mcal{L}$ for all $0\leq i\leq n$ where $\top$ (respectively $\bot$) is the greatest (respectively least) element of $L$. 

The two most important dimensions are m-dimension and breadth. The m-dimension, respectively breadth, of a modular lattice $L$ is $\mcal{L}\text{-}\dim L$ where $\mcal{L}$ is the class of lattices of size $2$, respectively the class of total orders. If $\pp_R^1(\mcal{C})$ has m-dimension then it is equal to the Cantor-Bendixson rank of $\Zg(\mcal{C})$ and, if $R$, is countable then the two values are always equal. The m-dimension of $\pp_R^1$ is equal to the Krull-Gabriel dimension, in the sense of \cite{Geigle}, of the category $(\mod\text{-}R,\Ab)^{fp}$ of finitely presented functors from $\mod\text{-}R$,
the category of finitely presented right $R$-modules, to $\Ab$, the category of abelian group. This dimension measures the complexity of factorisation of morphisms in $\mod\text{-}R$. If $\pp_R(\mcal{C})$ has breadth then there are no superdecomposable pure-injective modules in $\mcal{C}$ and, if $R$ is countable, the converse holds.

To formally define the $\mcal{L}$-dimension (see \cite[Section 10.2]{PrestBluebook} and \cite[Section 7.1]{PSL}) of a bounded modular lattice $L$ one defines a sequence of congruence relations $\sim_\alpha$ on $L$ for each ordinal $\alpha$ which is the congruence relation induced on $L$ by iteratively collapsing intervals in $\mcal{L}$ $\alpha$ times. The $\mcal{L}$-dimension of $L$ is then the least ordinal $\alpha$ such that $L/\sim_{\alpha+1}$ is the one element lattice.

\begin{proposition}\label{bounds}
Let $\mcal{L}$ be a non-empty class of lattices closed under sublattices and quotient lattices. Let $\mcal{C}\subseteq \Mod\text{-}R$ and $\mcal{D}\subseteq \Mod\text{-}S$ be definable subcategories and let $I:\mcal{C}\rightarrow \mcal{D}$ be an interpretation functor with $\langle I\mcal{C}\rangle=\mcal{D}$. If $I$ is full on pure-injectives then
\[\sup\{\mcal{L}\text{-}\dim \pp_S^1(\mcal{D}),\mcal{L}\text{-}\dim \pp_R^1(\ker I)\}\leq \mcal{L}\text{-}\dim\pp_R^1(\mcal{C})\leq\mcal{L}\text{-}\dim\pp_S^1(\mcal{D})+1+\mcal{L}\text{-}\dim \pp_R^1(\ker I). \] So, in particular, $\mcal{L}\text{-}\dim \pp^1_R(\mcal{C})$ exists if and only if $\mcal{L}\text{-}\dim\pp_R^1\ker I$ and $\mcal{L}\text{-}\dim\pp_S^1(\mcal{D})$ exist.
\end{proposition}

\begin{lemma}\label{latticeiso}
Let $\mcal{C}\subseteq \Mod\text{-}R$ and $\mcal{D}\subseteq \Mod\text{-}S$ be definable subcategories and let $I:\mcal{C}\rightarrow \mcal{D}$ be an interpretation functor with underlying pp-pair $\nf{\phi}{\psi}$. For each $\sigma\in \pp_S^n$, there exists a unique $\sigma'\in [\psi,\phi]_\mcal{C}$ such that for all $M\in \mcal{D}$, $\sigma(IM)=\sigma'(M)/\psi(M)$. The assignment $\sigma\mapsto \sigma'$ gives a lattice homomorphism $\pp_S^1(\mcal{D})\rightarrow [\psi,\phi]_\mcal{C}$. If $\langle I\mcal{C}\rangle=\mcal{D}$ then this lattice homomorphism is injective and if $I$ is full on pure-injectives then it is surjective.
\end{lemma}
\begin{proof}That such a $\sigma'$ exists follow easily from the definition of an interpretation functor. Uniqueness and that the assignment $\sigma\mapsto \sigma'$ gives a lattice homomorphism follows directly from the defining property of $\sigma'$. Suppose $\langle I\mcal{C}\rangle=\mcal{D}$ and $\sigma>\tau$ in $\pp_S^1(\mcal{D})$. Since $\langle I\mcal{C}\rangle=\mcal{D}$, there exists $M\in\mcal{C}$ such that $\sigma(IM)\supsetneq\tau(IM)$. Thus $\sigma'(M)\supsetneq \tau'(M)$. So the lattice homomorphism is injective. The final claim is a consequence of \cite[3.17]{Intmodinmod}.
\end{proof}

In order to prove \ref{bounds} we need a folklorish result which we were unable to find an explicit reference for. It can be proved using model theoretic methods or using functor categories methods, it follows from \cite[12.3.18 \& 13.1.4]{PSL}.

\begin{lemma}\label{defsubcatformdim}
Let $\mcal{L}$ be a non-empty class of lattices closed under sublattices and quotient lattices. Let $\mcal{E}$ be a definable subcategory of $\Mod$-$R$. For $\alpha$ an ordinal, let $\sim_\alpha$ be the congruence relation on $\pp_R^1(\mcal{E})$ induced by iteratively collapsing intervals in $\mcal{L}$ $\alpha$ times. Let 
\[\mcal{E}_\alpha:=\{M\in \mcal{E} \st \sigma(M)=\tau(M) \text{ for all }\sigma,\tau \text{ with }\sigma\sim_\alpha \tau\}.\] Then, for all $\sigma,\tau\in \pp_R^1$, $\sigma\sim_\alpha\tau$ if and only if $\sigma(M)=\tau(M)$ for all $M\in \mcal{E}_\alpha$. 
\end{lemma}

\begin{proof}[Proof of \ref{bounds}]
Let $\nf{\phi}{\psi}$ be the underlying pp-pair of $I$.
By \ref{latticeiso}, the interpretation functor $I$ induces a lattice isomorphism from $\pp_S^1(\mcal{D})$ to $[\psi,\phi]_\mcal{C}$. 

The lattice $\pp_R^1\ker I$ is a quotient of $\pp^1_R(\mcal{C})$ and the lattice $[\psi,\phi]_\mcal{C}\cong\pp_S^1(\mcal{D})$ is an interval of $\pp^n_R(\mcal{C})$. Hence if $\mcal{L}\text{-}\dim \pp^1_R(\mcal{C})=\mcal{L}\text{-}\dim \pp^n_R(\mcal{C})$ exists then so does $\mcal{L}\text{-}\dim\pp_R^1(\ker I)$ and $\mcal{L}\text{-}\dim\pp_S^1(\mcal{D})$. Moreover, the left hand inequality holds.

Using the notation of \ref{defsubcatformdim}, since $\mcal{L}\text{-}\dim [\psi,\phi]_\mcal{C}=\alpha$, we have that $\psi\sim_{\alpha+1}\phi$. Hence $\mcal{C}_{\alpha+1}\subseteq \ker I$. Therefore the quotient map $\pp_R^1(\mcal{C})\rightarrow \pp_R^1(\mcal{C})/\sim_{\alpha+1}$ factors through the quotient map $\pp_R^1(\mcal{C})\rightarrow \pp_R^1(\ker I)$. So the right hand equality holds. 
\end{proof}

\section{Applications to $R$-torsion-free modules over $R$-orders}\label{app}

Throughout this section, we fix the following assumptions and notation.
Let $R$ be a complete discrete valuation domain with maximal ideal generated by $\pi$ and field of fractions $Q$. Let $\Lambda$ be an $R$-order with $A:=Q\Lambda$ a separable $Q$-algebra. Recall that an $R$-order is an $R$-algebra which is torsion-free and finitely generated as an $R$-module. Let $\Latt_\Lambda$ denote the category of right $\Lambda$-lattices, i.e. the category of finitely generated $\Lambda$-modules which are torsion-free as $R$-modules, and let $\Tf_\Lambda$ denote the category of $R$-torsion-free $\Lambda$-modules.
\subsection{The functor of Ringel and Roggenkamp}
Under the hypotheses of this section, by \cite[26.5]{CurtisReiner}, there always exists a hereditary $R$-order $\Lambda\subseteq \Gamma\subseteq A$. Let $I\subseteq \rad(\Lambda)$ be an ideal of $\Gamma$ such that $\pi^n\in I$ for some $n\in\N$\footnote{This final condition appears to be missing in \cite{RR}. However, it is obviously necessary in order for $R_I$ to be Artinian: Let $\Lambda=R\times R$ then $\rad(\Lambda)=\mfrak{m}\times\mfrak{m}$ and $\mfrak{m}\times\{0\}\subseteq \rad(\Lambda)$.}. Let \[D:=                                                          \begin{pmat}
                                                            \Lambda/I & \Gamma/I \\
                                                            0 & \Gamma/I \\
                                                          \end{pmat}
                                                        .\]
Note that $D$ is an $R_I$-Artin algebra where $R_I:=R/R\cap I$.

We will be particularly interested in the case of \textbf{B\"ackstr\"om orders}, that is, when $\Gamma\supsetneq \Lambda$ and $I=\rad(\Lambda)=\rad(\Gamma)$. In this case, $D$ is a hereditary $R/\pi R$-algebra since $\Lambda/\rad(\Lambda)$ and $\Gamma/\rad(\Gamma)$ are both semisimple.

The category of right $D$-modules is equivalent to the category of triples $(U,V,f)$ where $U\in\Mod\text{-}\Lambda/I$, $V\in\Mod\text{-}\Gamma/I$ and $f:U\rightarrow V$ is a homomorphism of $\Lambda/I$-modules. Morphisms are given by $(\alpha,\beta):(U_1,V_1,f_1)\rightarrow (U_2,V_2,f_2)$ where $\alpha:U_1\rightarrow U_2$ is a $\Lambda/I$ homomorphism, $\beta:V_1\rightarrow V_2$ is a $\Gamma/I$ homomorphism and the diagram
\[\xymatrix@=20pt{
  U_1 \ar[d]_{\alpha} \ar[r]^{f_1} & V_1\ar[d]^{\beta} \\
  U_2 \ar[r]^{f_2} & V_2   }\] commutes.

Under this equivalence a triple $(U,V,f)$ is mapped to the $D$-module with underlying abelian group $U\times V$ and $D$-action given by
\[(u,v)\cdot
         \begin{pmat}
           a & b \\
           0 & c \\
         \end{pmat}
=(ua,f(u)b+vc).
\]
We will generally identify $\Mod\text{-}D$ with the category of triples.

For any $M\in \Tf_\Lambda$, define $M\Gamma$ to be the $\Gamma$-module generated by $M$ inside $MQ$. Then the inclusion of $M$ in $M\Gamma$ induces an embedding of $\sigma_M:M/MI\rightarrow M\Gamma/MI$.
Let
\[\mathbb{F}:\Tf_\Lambda\rightarrow \Mod\text{-}D, \ M\mapsto (M/MI, M\Gamma/MI,\sigma_M).\]

This is the natural extension to $\Tf_\Lambda$ of the functor of Ringel and Roggenkamp defined in \cite{RR} as a functor from $\Latt_\Lambda$ to $\mod\text{-}D$\footnote{In \cite{RR}, the functor $\mathbb{F}$ is also studied when $\Gamma$ is not hereditary. However, it is then restricted to the $\Lambda$-lattices $M$ such that $M\Gamma$ is projective.}.

We now give a description of $\mathbb{F}$ in terms of pp formulae. So, in particular, we show that $\mathbb{F}$ is an interpretation functor. Every ideal $K\lhd \Lambda$ is finitely generated as an $R$-submodule. If $K=\sum_{i=1}^m\alpha_iR$ then let $\Theta_K(x)$ be the pp formula $\exists \mathbf{y} \ x=y_1\alpha_1+\ldots +y_m\alpha_n$. For any $M\in\Mod\text{-}\Lambda$, $\Theta_K(M)$ is $MK$. Since $\Lambda\subseteq\Gamma\subseteq Q\Lambda$ and $\Gamma$ is an $R$-order, there exists $n\in \N$ such that $\pi^n\Gamma\subseteq \Lambda$. Then $\pi^n\Gamma$ is a finitely generated ideal of $\Lambda$. Therefore the functor $\mathbb{F}$ is given by $(\phi/\psi;(\rho_\delta)_{\delta\in D})$ where
\begin{itemize}
\item $\phi$ is $\Theta_{\pi^n\Lambda}(x_1)\wedge \Theta_{\pi^n\Gamma}(x_2)$;
\item $\psi$ is $\Theta_{\pi^nI}(x_1)\wedge\Theta_{\pi^nI}(x_2)$; and
\item for each $\delta:=\begin{pmat}a+I & b+I\\ 0 & c+I\end{pmat}$ where $a\in \Lambda$ and $b,c\in \Gamma$,
\[\rho_{\delta}(\mathbf{x},\mathbf{y}) \text{ is } y_1=x_1a\wedge y_2=x_1b+x_2c.\]
\end{itemize}

Let $\mcal{D}$ denote the full subcategory of $D$-modules $(U,V,f)$ such that $f$ is a monomorphism and $\im f$ generates $V$ as a $\Gamma/I$-module.

\begin{remark}
The category $\mcal{D}$ is a definable subcategory of $\Mod\text{-}D$.
\end{remark}
\begin{proof}
This is particularly easy to see in terms of pp formulae. Note $f$ is a monomorphism if and only if the subset defined by $x\cdot \begin{pmat}  0 & 1 \\ 0 & 0\\ \end{pmat}=0$ is $\{0\}$. Let $b_1,\ldots, b_n$ generate $\Gamma/I$ as an $R_I$-module. Then $\im f$ generates $V$ if and only if the pp formulae $\exists y \ y\cdot\begin{pmat}  0 & 0 \\ 0 & 1\\ \end{pmat}=x$ and $\exists y_1,\ldots,y_n  \ x= \sum_{i=1}^n y_i\cdot\begin{pmat}  0 & b_i \\ 0 & 0\\ \end{pmat}$ define the same subset in the $D$-modules corresponding to $(U,V,f)$.
\end{proof}

%

\begin{theorem}\label{RRFun}\cite[1.3 \& 1.4]{RR}
The functor $\mathbb{F}|_{\Latt_\Lambda}$ is full and for all $M\in \mcal{D}\cap\mod\text{-}D$ there exists $M'\in \Latt_\Lambda$ such that $\mathbb{F}M'\cong M$.
\end{theorem}

The next remark follows from \cite[5.22 (iv)]{CurtisReiner}.

\begin{remark}
There exists $m\in\N$ such that $\rad(\Lambda)^m\subseteq \pi\Lambda$. Hence, $I^m\subseteq \pi\Lambda$.
\end{remark}

For any $L\in \mod\text{-}\Lambda$, the canonical surjection $\pi_L:L\rightarrow L/\text{tor}(L)$ is a left $\Latt_\Lambda$-approximation of $L$, where $\text{tor}(L)$ denotes the $\Lambda$-submodule of $R$-torsion elements of $L$. Therefore, the category $\Latt_\Lambda$ is covariantly finite in $\mod\text{-}\Lambda$.

\begin{theorem}\label{RRpureinj}
The functor $\mathbb{F}:\Tf_\Lambda\rightarrow \Mod\text{-}D$ is full on pure-injectives, $\ker \mathbb{F}$ is the class of $R$-divisible $R$-torsion-free $\Lambda$-modules and $\langle \mathbb{F}\Tf_\Lambda\rangle =\mcal{D}$.
\end{theorem}
\begin{proof}
 The first statement follows from \ref{fullonfpimpliesfullonpureinjectives} and \ref{RRFun}. Let $M\in\Tf_\Lambda$. Then $\mathbb{F}M=0$ if and only if $M/MI=0$. Since $I^m\subseteq \pi\Lambda$ and $\pi^l\in I$ for some $l,m\in\N$, $M=MI$ if and only if $M\pi=M$. Therefore $M\in \ker \mathbb{F}$ if and only if $M$ is $R$-divisible. For the final claim, it is enough to show that every $M\in \mcal{D}$ is a directed union of finitely presented modules in $\mcal{D}$. It follows from the definition of $\mathbb{F}$ that $\mathbb{F}M\in \mcal{D}$ for all $M\in\Tf_\Lambda$. Suppose $M:=(U,V,f)\in\mcal{D}$ and $M_0:=(U_0,V_0,f_0)$ is finitely presented (equivalently finitely generated) submodule of $M$. Then $f_0$ is injective. Take $u_1,\ldots,u_n\in U$ such that the $\Gamma/I$-submodule generated by $f(u_1),\ldots f(u_n)$ contains $V_0$. Let $U_1$ be the $\Lambda/I$-submodule generated by $U_0$ and $u_1,\ldots,u_n$ and let $V_1$ be the $\Gamma/I$-submodule generated by $f(u_1),\ldots f(u_n)$. Then $(U_1,V_1,f|_{U_1})$ is finitely generated and contains $M_0$.
\end{proof}

%
%
%
\subsection{Pure-injective $R$-torsion-free $\Lambda$-modules and the torsion free part of the Ziegler spectrum}\label{ZgFgeneral}

The torsion free part of $\Zg_\Lambda$, denoted $\Zg_\Lambda^{tf}$, is the closed set $\Zg(\Tf_\Lambda)$ of $\Zg_\Lambda$. This space was studied in \cite{tfpartRG}, \cite{Klein4}, \cite{Tfpart} and \cite{Maranda}. In this section, we first briefly review the general features of this space and then show how \ref{RRpureinj} can be used to understand the topology on $\Zg_\Lambda^{tf}$ in terms of the topology on $\Zg_D$.

The $R$-divisible $R$-torsion-free $\Lambda$-modules are a definable subcategory equivalent to $\Mod\text{-}Q\Lambda$. Since $Q\Lambda$ is separable, and hence semisimple, all $R$-divisible modules are injective (see \cite[Claim 2 p. 1128]{tfpartRG}) and hence pure-injective as $\Lambda$-modules. An $R$-torsion-free $\Lambda$-module $M$ is \textbf{$R$-reduced} if $\cap_{i\in\N}M\pi^i=\{0\}$.
For any $M\in\Tf_\Lambda$,
\[D_M:=\{m\in M \st \pi^n|m \text{ for all }n\in\N\}\] is a divisible submodule of $M$ and $M/D_M$ is $R$-reduced. Since $D_M$ is injective, $M\cong D_M\oplus M/D_M$.

Since we assume that $R$ is complete, $\Lambda$-lattices are pure-injective, see for example \cite[2.2]{Tfpart}. The indecomposable $\Lambda$-lattices are isolated in $\Zg_\Lambda^{tf}$ \cite[2.4]{Tfpart}. The set of indecomposable $\Lambda$-lattices are dense in $\Zg_\Lambda^{tf}$ \cite[2.2]{Tfpart} and hence, the previous sentence implies that they are exactly the isolated points. By the discussion above, all $N\in\Zg_\Lambda^{tf}$ are either $R$-divisible or $R$-reduced. There are finitely many indecomposable $R$-divisible $R$-torsion-free $\Lambda$-modules; they are the (restrictions to $\Lambda$ of the) simple $Q\Lambda$-modules. These are exactly the closed points of $\Zg_\Lambda^{tf}$ (this follows from \cite[2.8]{Tfpart}). If $\Lambda$ is of finite lattice type then, \cite[5.4]{Tfpart}, all points in $\Zg_\Lambda^{tf}$ are either $\Lambda$-lattices or $R$-divisible. The $R$-reduced modules in $\Zg_\Lambda^{tf}$ are those modules in the open set $\left(\nf{x=x}{\pi|x}\right)$ and the complement of this set may be identified $\Zg_{Q\Lambda}$.

For each simple $Q\Lambda$-module $S$ there exists an open subset $\mcal{V}(S)$ such that for all $N\in\Zg_\Lambda^{tf}$, $S$ is in the closure of $N$ if and only if $N\in\mcal{V}(S)$. Moreover $N\in\mcal{V}(S)$ if and only if $S$ is a direct summand of $NQ$. Based on this fact we get the following description of open subsets of $\Zg_{\Lambda}^{tf}$.

\begin{lemma}\cite[5.11]{Maranda}
Let $U$ be an open subset of $\Zg_\Lambda^{tf}$. Then
\[U=(U\cap\left(x=x/\pi|x\right))\cup\bigcup_{S\in\lambda(U)}\mcal{V}(S)\] where $\lambda(U):=\{S\in\Zg_{Q\Lambda} \st S\in U\}$.
\end{lemma}

Using \ref{fpiZghomeo}, we get the following corollary to \ref{RRpureinj}.

\begin{cor}\label{RRhomeomorphism}
The functor $\mathbb{F}:\Tf_\Lambda\rightarrow \Mod\text{-}D$ induces a homeomorphism between the open set $\left(\nf{x=x}{\pi|x}\right)$ of $R$-reduced points in $\Zg^{tf}_\Lambda$ and the closed subset $\Zg(\mcal{D}):=\Zg_D\cap\mcal{D}$.
\end{cor}

This corollary allows us to understand the topology on $\Zg_\Lambda^{tf}$ in terms of the topology on $\Zg_D$.

\begin{definition}
For $V$ a closed subset of $\Zg_D$, define 
\[V_0:=\{N\in \left(\nf{x=x}{\pi|x}\right) \st \mathbb{F}N\in V\}\] and define $\overline{V}$ to be the closure of $V_0$ in $\Zg_\Lambda^{tf}$.
\end{definition}

\begin{proposition}\label{topFdesc}
Let $V$ be a closed subset of $\Zg_D$. Then \[\overline{V}=V_0\cup\{S\in\Zg_{Q\Lambda}\st \text{ there exists }N\in \left(x=x/\pi|x\right) \text{ with } S|NQ \text{ and } \mathbb{F}N\in V\}.\] Moreover, all closed subsets of $\Zg^{tf}_\Lambda$ are of the form $\overline{V}\cup W$ where $V\subseteq \Zg_D$ is a closed subset and $W\subseteq \Zg_{Q\Lambda}$.
\end{proposition}
\begin{proof}
Since $V$ is closed, so is $V_1:=\{N\in \Zg^{tf}_\Lambda \st \mathbb{F}N\in V\}$. Let $X$ be the set of $S\in \Zg_{Q\Lambda}$ such that $\mcal{V}(S)\cap V_1\subseteq \Zg_{\Lambda Q}$. Then $\overline{V}\subseteq V_1\cap\bigcap_{S\in X}(\Zg_\Lambda^{tf}\backslash\mcal{V}(S))$. Suppose that $N\in V_1\cap\bigcap_{S\in X}(\Zg_\Lambda^{tf}\backslash\mcal{V}(S))$. If $N\in \left(\nf{x=x}{\pi|x}\right)$ and $N\in V_1$ then $N\in V_0\subseteq \overline{V}$. If $N\in V_1\cap\bigcap_{S\in X}(\Zg_\Lambda^{tf}\backslash\mcal{V}(S))$ and $N\in\Zg_{Q\Lambda}$ then $\mcal{V}(N)\cap V_1$ contains an element $N'$ in $V_0$. By definition of $\mcal{V}(N)$, $N$ is in the closure of $N'$. For any closed set $C\subseteq \Zg_\Lambda^{tf}$, $\overline{C\cap \left(\nf{x=x}{\pi|x}\right)}$ is equal to $\overline{V}$ for some $V\subseteq \Zg_{D}$ and $C\backslash\overline{C\cap \left(\nf{x=x}{\pi|x}\right)}\subseteq \Zg_{Q\Lambda}$.
\end{proof}

\subsection{Pseudogeneric modules}\label{SecPsG}
A module over an arbitrary ring $S$ is \textbf{generic} if it is indecomposable, of finite length over its endomorphism ring and not finitely presented. A module $M$ is finite length over its endomorphism ring if and only if $\pp_S^n(M)$ is finite length for some, equivalently all, $n\in\N$ (see \cite[4.4.25]{PSL}). This implies that $M$ is pure-injective.
An $R$-torsion-free $\Lambda$-module is finite length over its endomorphism ring if and only if it is $R$-divisible. This is because for all $n\in\N$, $M\pi^n$ is an $\End(M)$-submodule of $M$ and since $M$ is $R$-torsion-free $M\pi^n=M\pi^{n+1}$ implies $M\pi=M$. Although the indecomposable divisible modules are important, they do not play the same role for lattices over orders as generic modules play for finitely presented modules over Artin algebras. For instance, \cite[Thm 9.6]{CBfinendo}, if $S$ is an Artin algebra then there exists a generic $S$-module if and only if for some $n\in\N$, there are infinitely many non-isomorphic indecomposable $S$-modules of endolength $n$. However, there are always finitely many indecomposable $R$-divisible $R$-torsion-free $\Lambda$-modules.\looseness=-1

We propose the notion of a pseudogeneric $R$-torsion-free $\Lambda$-module as a replacement for the generic modules over Artin algebras. To motivate this definition we first need to summarise some results around Maranda's theorem. Since $Q\Lambda$ is separable, by \cite[29.5]{CurtisReiner} and the discussion just after \cite[30.12]{CurtisReiner}, there exists $k\in \N$ such that for all $L,M\in\Latt_\Lambda$, $\pi^k\Ext(L,M)=0$. Let $k_0$ be the least such $k$. Under the hypothesis of this section, Maranda's theorem, \cite[30.14 \& 30.19]{CurtisReiner}, states that for any $k\geq k_0+1$, a pair of $\Lambda$-lattices $M,N$ are isomorphic if and only if $M/M\pi^k\cong L/L\pi^k$. Moreover, if $M$ is an indecomposable $\Lambda$-lattice then so is $M/M\pi^k$. A version of this result holds for $R$-reduced $R$-torsion-free pure-injective $\Lambda$-modules.

\begin{theorem}\cite[3.4, 3.5 \& 5.1]{Maranda}\label{Marandapi}
Let $N,M\in\Tf_\Lambda$ be $R$-reduced and pure-injective and let $k\geq k_0+1$. If $N$ is indecomposable then $N/N\pi^k$ is indecomposable and if $N/N\pi^k\cong M/M\pi^k$ then $N\cong M$. Moreover, the map which sends $N\in\left(x=x/\pi|x\right)\subseteq \Zg_\Lambda^{tf}$ to $N/N\pi^k\in \Zg_{\Lambda_k}$ induces a homeomorphism onto its image which is closed.
\end{theorem}

For $k\geq k_0+1$, we will refer to the functor taking $M\in \Tf_\Lambda$ to $M/M\pi^k$ as Maranda's functor. When the exact value of $k\geq k_0+1$ is not important we will write $\mathbb{J}$ for this functor. Although Maranda's functor is not full (even on pure-injectives) is preserves the structure on the interval $[\pi|x,x=x]$.   The next proposition is part of \cite[4.6]{Tfpart}.

\begin{proposition}\label{Marintiso}
The map from $[\pi+\Lambda\pi^k|x,x=x]_{\langle I\Tf_\Lambda\rangle}$ to $[\pi|x,x=x]_{\Tf_\Lambda}$ induced by Maranda's functor $\Tf_\Lambda\rightarrow \Mod\text{-}\Lambda_k$ is an isomorphism. In particular, for any $M\in\Tf_\Lambda$, the induced map from $[\pi+\Lambda\pi^k|x,x=x]_{M/M\pi^k}$ to $[\pi|x,x=x]_M$ is surjective.
\end{proposition}

We say $M\in\Tf_\Lambda$ is \textbf{pseudogeneric} if $M$ is indecomposable, pure-injective, $R$-reduced and $M/M\pi^{k_0+1}$ is finite-length over its endomorphism ring. For $M\in \Tf_\Lambda$, we define the \textbf{pseudoendolength} of $M$ to be the length of $[\pi|x,x=x]_M$. Note that this means that $R$-divisible $\Lambda$-modules have pseudoendolength $0$.\looseness=-1

\begin{lemma}
Let $M\in\Tf_\Lambda$ and let $k\geq k_0+1$.
\begin{enumerate}
\item The endolength of $M/M\pi^k$ is equal to $k$ times the pseudoendolength.
\item The module $M$ is pseudogeneric if and only if it is indecomposable, pure-injective, $R$-reduced and has finite pseudoendolength.
\end{enumerate}
\end{lemma}
\begin{proof}
The second statement follows directly from the first. To prove $(1)$, take $M\in \Tf_\Lambda$. Let $p:=\pi+\Lambda\pi^k$. For each $1\leq i\leq k-1$, $[p^{i+1}|x,p^i|x]_{M/M\pi^k}$ is isomorphic to $[p|x,x=x]_{M/M\pi^k}$. Hence, since $p^k|x$ is equivalent to $x=0$ in $M/M\pi^k$, the length of $\pp_{\Lambda/\Lambda\pi^k}(M/M\pi^k)$ is equal to $k$ times the length of $[p|x,x=x]_{M/M\pi^k}$. By \ref{Marintiso}, $[p|x,x=x]_{M/M\pi^k}$ is isomorphic to $[\pi|x,x=x]_M$. Therefore $k$ times the pseudoendolength is equal to the length of $\pp_{\Lambda/\Lambda\pi^k}(M/M\pi^k)$. By \cite[4.4.25]{PSL}, the length of $\pp_{\Lambda/\Lambda\pi^k}(M/M\pi^k)$ is equal to the endolength of $M/M\pi^k$.
\end{proof}


Indecomposable modules of finite length over their endomorphism ring are closed points of the Ziegler spectrum, \cite[5.1.12]{PSL}. Thus by \ref{Marandapi} and \cite[2.2 \& 2.4]{Tfpart} we get the following remark.
\begin{remark}
If $M\in\Tf_\Lambda$ is pseudogeneric then it is closed and not isolated in the open set $\left(x=x/\pi|x\right)$ equipped with the subspace topology.
\end{remark}


\begin{lemma}\label{psgentopprop}
Suppose that one of the following conditions hold.
\begin{enumerate}
\item $R/R\pi$ is countable.
\item The isolation condition holds for $\Lambda/\Lambda\pi^{k_0+1}$.
\item $\pp_\Lambda^{1}(\Tf_\Lambda)$ has m-dimension.
\end{enumerate}
An indecomposable pure-injective module $M$ is pseudogeneric if and only if it is closed and not isolated in the open set $\left(x=x/\pi|x\right)$ equipped with the subspace topology.
\end{lemma}
\begin{proof}
The forwards direction is always true. Recall that the isolated points of $\Zg_\Lambda^{tf}$ are $\Lambda$-lattices and that $N\in \left(x=x/\pi|x\right)$ if and only if $N$ is $R$-reduced. Thus, under the hypothesis that (1), (2) or (3) holds, we just need to show that if $N$ is closed in $\left(x=x/\pi|x\right)$ then $N$ has finite pseudoendolength. If $R/R\pi$ is countable then so is $\Lambda/\Lambda\pi^{k_0+1}$. Thus $(1)$ implies $(2)$. Suppose that $(2)$ holds and $N\in \left(x=x/\pi|x\right)$ is closed in $\left(x=x/\pi|x\right)$. Then $N/N\pi^{k_0+1}$ is closed in $\Zg_{\Lambda/\Lambda\pi^{k_0+1}}$. Since the isolation condition holds, by \cite[5.3.23]{PSL}, $N/N\pi^{k_0+1}$ is of finite endolength. Hence $N$ is of finite pseudoendolength. If $(3)$ holds then, by \cite[5.3.17]{PSL}, the isolation condition holds for $\Zg_\Lambda^{tf}$. Hence \cite[4.3.22]{PSL} implies that if $N$ is closed in $\left(x=x/\pi|x\right)$ then $[\pi|x,x=x]_N$ is finite length as required.\looseness=-1
\end{proof}

The proof of the next proposition copies Herzog's proof, \cite[9.6]{Zgspecloccoh}, of the analogous result for Artin algebras.

\begin{proposition}
If there are infinitely many non-isomorphic indecomposable $\Lambda$-lattices of pseudoendolength $l$ then there exists a pseudogeneric of pseudoendolength $\leq l$.
\end{proposition}
\begin{proof}
By \cite[9.3]{Zgspecloccoh}, for any ring $S$, pp-pair $\phi/\psi$ and $n\in\N_0$, the set of $N\in\Zg_S$ such that $[\psi,\phi]_N$ has length $\leq n$ is a closed subset of $\Zg_S$.
Therefore the set $\mcal{C}_l$ of $N\in \Zg_\Lambda^{tf}$ with pseudoendolength $\leq l$ is a closed subset (note that this set also contains the divisible points). Since $\left(x=x/\pi|x\right)$ is compact, it follows that $\left(x=x/\pi|x\right)\cap \mcal{C}_l$ is compact. The indecomposable $\Lambda$-lattices are isolated in $\Zg_\Lambda^{tf}$. Thus if there are infinitely many non-isomorphic indecomposable $\Lambda$-lattices of pseudoendolength $l$ then $\mcal{C}_l \cap\left(x=x/\pi|x\right)$ contains infinitely many isolated points and hence must contain a non-isolated point $N$.
\end{proof}


Let $e_1:=\begin{pmat}1 & 0 \\ 0 & 0
\end{pmat}, e_2:=\begin{pmat}0 & 0 \\ 0 & 1\end{pmat}\in D$. For any $N\in\Mod\text{-}D$, $\pp_D^1(N)$ is finite length if and only if $[x=0,e_1|x]_N$ and $[x=0,e_2|x]_N$ are finite length because \[[e_1|x,x=x]_N=[xe_2=0,x=x]_N\cong[x=0,e_2|x]_N.\]

\begin{proposition}\label{pseudogenF}
If $N$ is pseudogeneric then $\mathbb{F}N$ is a generic module. If any of the conditions $(1),(2)$ or $(3)$ from \ref{psgentopprop} hold then if $N\in \Tf_\Lambda$ is $R$-reduced, pure-injective and $\mathbb{F}N$ is generic then $N$ is pseudogeneric.
\end{proposition}
\begin{proof}
Suppose $N$ is pseudogeneric. Let $n\in\N$ be such that $\pi^n\in I$. Since $N$ is pseudogeneric the interval $[\pi|x,x=x]_N$ is finite length. Therefore, the interval $[\pi^l|x,x=x]_N$ is finite length for any $l\in \N$. Since $\pi^n\in I$, $\pi^n|x\leq I|x$. Hence $[I|x,x=x]_N$ is finite length. Let $m\in\N$ be such that $\pi^m\Gamma\subseteq \Lambda$. Then $[\pi^mI|x,x=x]_N$ and hence $[\pi^mI|x,\pi^m\Gamma]_N$ are finite length. Since $\mathbb{F}$ is full on pure-injectives, by \cite[3.17]{PSL}, $[I|x,x=x]_N$ is isomorphic to $[x=0,e_1|x]_{\mathbb{F}N}$ and $[\pi^mI|x,\pi^m\Gamma]_N\cong [x=0,e_2|x]_{\mathbb{F}N}$. Thus $\pp_D(\mathbb{F}N)$ is finite length. By \ref{RRpureinj}, $\mathbb{F}N$ is indecomposable. Every finitely presented module in $\mcal{D}$ is of the form $\mathbb{F}M$ for some $M\in\Latt_\Lambda$. Since $R$ is complete, $\Lambda$-lattices are pure-injective. So, if $\mathbb{F}N$ were finitely presented then $N\cong M$ for some $\Lambda$-lattice $M$. Thus we have proved the first statement.

Suppose that one of $(1),(2)$ or $(3)$ from \ref{psgentopprop} holds. Since $\mathbb{F}N$ is generic, it is closed in $\Zg(\mcal{D})$. So, by \ref{RRhomeomorphism}, $N$ is closed in $\left(\nf{x=x}{\pi|x}\right)$. Since $\mathbb{F}N$ is generic, $N$ is not a $\Lambda$-lattice. Thus $N$ is not isolated in $\left(\nf{x=x}{\pi|x}\right)$. So, by \ref{psgentopprop}, $N$ is pseudogeneric.
\end{proof}

We now specialise to the case of B\"{a}ckstr\"{o}m orders i.e. an order $\Lambda$ where there exists a hereditary over order $\Gamma\neq \Lambda$ such that $\rad(\Lambda)=\rad(\Gamma)$. In this situation we set $I=\rad(\Lambda)$. Now $\Lambda/I$ and $\Gamma/I$ are both semi-simple $R/\pi R$-algebras. Hence $D$ is a hereditary $R/\pi R$-algebra. 

\begin{remark}\label{pinjinD}
For $\Lambda$ a B\"{a}ckstr\"{o}m order, the only indecomposable pure-injective $D$-modules which are not in $\mcal{D}$ are the simple $D$-modules.
\end{remark}
%
\begin{proof}
The statement follows for indecomposable finite-dimensional $D$-modules because $\Lambda/I$ and $\Gamma/I$ are semi-simple. Since $D$ has only finitely many simple modules and all finitely presented indecomposable $D$-modules are isolated in $\Zg_D$, $\langle \mathbb{F}\Tf_\Lambda\rangle$ is the definable subcategory of $\Mod\text{-}D$ containing all indecomposable non-simple pure-injective $D$-modules.
\end{proof}
%

Following Ringel and Roggenkamp, we say that a B\"{a}ckstr\"{o}m order $\Lambda$ is \textbf{tame} (respectively \textbf{wild}) if $D$ is tame (respectively wild). It is shown in \cite[5.1]{Tfpart} that when an arbitrary order $\Lambda$ over a Dedekind domain is of finite lattice type, i.e. there are only finitely many indecomposable $\Lambda$-lattices up to isomorphism, then $\textrm{m-dim}\, \pp^1_\Lambda(\Tf_\Lambda)=1$. 

\begin{proposition}\label{mdimBaeckstroem}
If $\Lambda$ is a wild B\"{a}ckstr\"{o}m order then $\textnormal{m-dim}\,\pp^1_\Lambda\Tf_\Lambda$ is undefined.
If $\Lambda$ is a tame B\"{a}ckstr\"{o}m order of infinite lattice-type then $\textnormal{m-dim}\,\pp^1_\Lambda\Tf_\Lambda$ is $3$. 
\end{proposition}

%
%

\begin{proof}
If $\Lambda$ is a wild B\"{a}ckstr\"{o}m order then the m-dimension of $\pp_D^1$ is undefined. This implies that $\pp_D^1(\mcal{D})$ also has undefined m-dimension because if $\sigma\geq \tau$ is such that the Ziegler basic open set $(\sigma/\tau)$ contains no infinite dimensional modules then $[\phi,\psi]$ is finite length. Thus, by \ref{latticeiso}, the m-dimension of $\pp_\Lambda^1\Tf_\Lambda$ is undefined.

Suppose that $\Lambda$ is a tame B\"{a}ckstr\"{o}m order of infinite lattice-type. Then, \cite[4.3]{Geigle}, the m-dimension of $\pp_D^1$ is $2$ and the m-dimension of $\pp_\Lambda^1\ker \mathbb{F}$ is $0$. Thus, by \ref{bounds},  $2\leq \textnormal{m-dim}\pp^1_\Lambda\Tf_\Lambda\leq 3$. Thus, \cite[5.3.60]{PSL}, the Cantor-Bendixson (CB) rank of $\Zg_\Lambda^{tf}$ is equal to the m-dimension of $\pp^1_\Lambda(\Tf_\Lambda)$. Let $N\in \Zg_D$ be a point of CB rank $2$. The finite-dimensional $D$-modules are isolated in $\Zg_D$. So $N$ also has CB rank $2$ in the subspace of $\Zg_D$ with the simple modules removed. Let $M\in\Zg_\Lambda^{tf}$ be such that $\mathbb{F}M=N$. By \ref{RRhomeomorphism}, $M$ has CB rank $2$ in $\left(x=x/\pi|x\right)$ and hence also in $\Zg_\Lambda^{tf}$. Let $S$ be a simple $Q\Lambda$-module which is a direct summand of $QM$. Then, by \cite[2.7]{Tfpart}, $S$ is in the closure of $N$. Thus $S$ has Cantor-Bendixson rank $3$ as required.  
\end{proof}

The previous lemma plus \ref{pseudogenF} gives that when $\Lambda$ is a tame B\"ackstr\"om order $\mathbb{F}$ gives a bijective correspondence between the generic modules of $D$ and the pseudogeneric modules of $\Lambda$. When $D$ is tame hereditary, it only has finitely many generic modules. 

\begin{cor}
If $\Lambda$ is a tame B\"ackstr\"om order of infinite representation type then $\Lambda$ has finitely many pseudogeneric $R$-torsion-free modules.
\end{cor}

\subsection{Tame B\"ackstr\"om orders}\label{tamebkzg}

Throughout this section we assume that $\Lambda$ is a tame B\"ackstr\"om order of infinite lattice type. By definition, $D$ is a tame hereditary algebra. Note that, even if we assume that $\Lambda$ is indecomposable, $D$ is not necessarily indecomposable (see \cite[Example 4]{RR}). Let $D=\prod_{i=1}^{n+1}D_i$ be such that $D_{n+1}$ is of finite representation type and that for $1\leq i\leq n$, $D_i$ is of infinite representation type and indecomposable as a ring. 

We briefly recall some information about the category of finite-dimensional $D$-modules (see \cite{DRGraph} and \cite{infringel}) and about the Ziegler spectrum of $D$ (see \cite{PrestTameher} and \cite{RingelZg}). Let $\tau$ denote the Auslander-Reiten translate which, since $D$ is hereditary, is an endofunctor on $\mod\text{-}D$. We call the finite-dimensional $D$-modules of the form $\tau^{-n}P$ where $P$ is a projective module and $n\in \N_0$ \textbf{preprojective}; finite-dimensional modules of the form $\tau^{n}I$ where $I$ is an injective module and $n\in \N_0$ \textbf{preinjective}; and finite-dimensional modules with no preprojective or preinjective direct summands \textbf{regular}. 
Note that an indecomposable finite-dimensional $D$-module is both preprojective and preinjective if and only if it is a $D_{n+1}$-module.

The (full) subcategory of regular modules is an abelian category and its inclusion into $\mod\text{-}D$ is exact. A $D$-module is \textbf{quasi-simple} if it is simple in the category of regular modules. We call a quasi-simple \textbf{type $i$} if it is a $D_i$-module. For each natural number $n$ and quasi-simple module $E$, there exists a unique regular module $E[n]$ which as an object in the category of regular modules is uniserial, has socle $E$ and length $n$. Dually, there exists a unique regular module $[n]E$ which as an object in the category of regular modules is uniserial, has top $E$ and length $n$. All indecomposable regular modules are of the form $E[n]$ (and dually of the form $[n]E$) for some quasi-simple $E$ and $n\in\N$.\looseness=-1


For each quasi-simple $E$, the direct limit along a chain of (irreducible) embeddings
\[E[1]\hookrightarrow E[2]\hookrightarrow E[3] \hookrightarrow\ldots\] is denoted $E[\infty]$. Dually, for each $E$, the inverse limit along a chain of (irreducible) epimorphisms
\[[1]E\leftarrow [2]E\leftarrow [3]E\leftarrow\ldots\] is denoted $\widehat{E}$. The modules $E[\infty]$ and $\widehat{E}$ are both pure-injective and indecomposable. There are finitely many remaining infinite dimensional indecomposable pure-injective $D$-modules. Namely, for each $1\leq i\leq n$, there is a unique indecomposable generic $D_i$-module. 

By \ref{pinjinD}, all infinite dimensional indecomposable pure-injective $D$-modules are in the image of $\mathbb{F}$. Thus $\mathbb{F}$ gives a bijective correspondence between the $R$-reduced modules in $\Zg_\Lambda^{tf}$ which are not $\Lambda$-lattices and the infinite dimensional points in $\Zg_D$. For each $1\leq i\leq n$, let $G_i$ denote the unique pseudogeneric module in $\Mod\text{-}\Lambda$ such that $\mathbb{F}G_i$ is the unique generic $D_i$-module \ref{pseudogenF}.

The topology on $\Zg_D$ for $D$ an indecomposable tame hereditary algebra is described in \cite{PrestTameher} and \cite{RingelZg}. We use the formulation from \cite{RingelZg}.

\begin{theorem}
A subset $C\subseteq \Zg_\Lambda^{tf}$ is closed if and only if the following conditions hold:
\begin{enumerate}
\item If $\mathbb{F}S$ is quasi-simple and if there exist infinitely many indecomposable $X\in \Latt_\Lambda$ with $\Hom_D(\mathbb{F}S,\mathbb{F}X)\neq 0$ then $\mathbb{F}^{-1}(\mathbb{F}S[\infty])\in C$.
\item If $\mathbb{F}S$ is quasi-simple and if there exist infinitely many indecomposable $X\in \Latt_\Lambda$ with $\Hom_D(\mathbb{F}X,\mathbb{F}S)\neq 0$ then $\mathbb{F}^{-1}(\widehat{\mathbb{F}S})\in C$.
\item If there are infinitely many indecomposable $X\in \Latt_\Lambda$ of type $i$ in $C$ or if there exists a least one module of the form $\mathbb{F}^{-1}(\mathbb{F}S[\infty])$ or $\mathbb{F}^{-1}(\widehat{\mathbb{F}S})$ where $\mathbb{F}S$ is a quasi-simple of type $i$ in $C$, then $G_i\in C$.
\item For each simple $Q\Lambda$-module $M$, if there exists $N\in C$ such that $M$ is a direct summand of $NQ$ then $M\in C$.
\end{enumerate}
\end{theorem}
\begin{proof}
If $C$ is closed then that $(1)$, $(2)$ and $(3)$ are true of $C$ follows from the description of the Ziegler spectra of tame hereditary algebras given in \cite{RingelZg}. Property (4) holds by \cite[2.7]{Tfpart}. 

Now suppose that $C$ is a subset of $\Zg_\Lambda^{tf}$ and $(1)$-$(4)$ are true of $C$. Then, since $(1)$-$(3)$ are true of $C$, the set $V$ of $\mathbb{F}N$ such that $N\in C$ and $N$ is not $R$-divisible is a closed subset of $\Zg_D$ by \ref{RRhomeomorphism}. It now follows from $(4)$ that $C$ is
\[\{N\in \left(\nf{x=x}{\pi|x}\right) \st \mathbb{F}N\in V\}\cup\{S\in\Zg_{Q\Lambda}\st \text{ there exists }N\in \left(\nf{x=x}{\pi|x}\right) \text{ with } S|NQ \text{ and } \mathbb{F}N\in V\}\cup W\] where $W$ is some subset of $\Zg_{Q\Lambda}$.  So by \ref{topFdesc}, $C$ is closed. 
\end{proof}

We now give a description of the modules $\mathbb{F}^{-1}(\mathbb{F}S[\infty])$ as $\Lambda$-modules. We are not able to give a description of the modules $\mathbb{F}^{-1}(\widehat{\mathbb{F}S})$ but one might guess that this module is isomorphic to the inverse limit along the irreducible epimorphisms $\mathbb{F}S[i+1]\rightarrow \mathbb{F}S[i]$. However, it is unclear whether this module is indecomposable.  

Roggenkamp, in \cite{ARBaeckstroem}, showed how to obtain the Auslander-Reiten quiver of $\Latt_\Lambda$ from the Auslander-Reiten quiver of $\mod\text{-}D$ using $\mathbb{F}$. It is a consequence of \cite[1.11(ii)]{ARBaeckstroem} that each of the stable tubes in the Auslander-Reiten quiver of $D$ lifts to a stable tube in the Auslander-Reiten quiver of $\Latt_\Lambda$. If $S\in\Latt_\Lambda$ is such that $\mathbb{F}S$ is quasi-simple then we write $S[n]$ (respectively $[n]S$) for the $L\in \Latt_\Lambda$ with $\mathbb{F}L=(\mathbb{F}S)[n]$ (respectively $\mathbb{F}L=[n](\mathbb{F}S)$).

For $M\in\Mod\text{-}\Lambda$, we write $M^\star$ for the inverse limit along the surjective maps $M/M\pi^{n+1}\rightarrow M/M\pi^n$ for $n\in\N$. Note that if $M$ is $R$-reduced then so is $M^\star$.

\begin{proposition}\label{pseudopruefer}
Let $S\in\Latt_\Lambda$ be such that $\mathbb{F}S$ is quasi-simple and let $S[\infty]$ denote the direct limit along a ray of irreducible maps $S[n]\rightarrow S[n+1]$. Then $S[\infty]^\star$ is an $R$-torsionfree $R$-reduced indecomposable pure-injective $\Lambda$-module and $\mathbb{F}(S[\infty]^\star)$ is isomorphic to $(\mathbb{F}S)[\infty]$.
\end{proposition}
\begin{proof}
It is explained in \cite{Maranda}, see \cite[4.9]{Maranda} and the discussion above the question on page 596 of that article, that $S[\infty]^\star$ is the pure-injective hull of $E[\infty]$. Since $\mathbb{F}$ commutes with direct limits and the irreducible maps $S[n]\rightarrow S[n+1]$ are sent to irreducible maps $(\mathbb{F}S)[n]\rightarrow (\mathbb{F}S)[n+1]$ \cite[1.12]{ARBaeckstroem}, $\mathbb{F}S[\infty]\cong (\mathbb{F}S)[\infty]$. Since $\mathbb{F}$ is full on pure-injectives and hence, \cite[3.15 \& 3.16]{Intmodinmod} preserves pure-injective hulls, $\mathbb{F}S[\infty]^\star\cong (\mathbb{F}S)[\infty]$.
Finally, if $S[\infty]^\star=N\oplus M$ then either $\mathbb{F}N=0$ or $\mathbb{F}M=0$. Hence, by \ref{RRpureinj}, either $N$ or $M$ is $R$-divisible. This contradicts that fact that $S[\infty]^\star$ is $R$-reduced.\looseness=-1
\end{proof}
\bibliographystyle{alpha}
\bibliography{Intfunfullonpinj}
\end{document}